\documentclass[12pt]{amsart}

\usepackage{bookmark}
\usepackage{enumerate}
\usepackage{tikz-cd}
\usepackage{comment}
\usepackage{amsrefs}
\usepackage{amsmath}
\usepackage{amssymb}
\usepackage{amsthm}
\usepackage{pgffor}

\newtheorem{theorem}{Theorem}[section]
\newtheorem{lemma}[theorem]{Lemma}
\newtheorem{corollary}[theorem]{Corollary}
\newtheorem{proposition}[theorem]{Proposition}
\newtheorem{fact}[theorem]{Fact}

\newtheorem{question}{Question}[section]

\theoremstyle{definition}

\newtheorem{definition}[theorem]{Definition}

\renewcommand{\phi}{\varphi}

\renewcommand{\epsilon}{\varepsilon}

\newcommand*\diff{\mathop{}\!\mathrm{d}}

\newcommand{\eps}{\varepsilon}
\foreach \x in {A,...,Z}{%
\expandafter\xdef\csname f\x\endcsname{\noexpand\ensuremath{\noexpand\mathfrak{\x}}}
}

\foreach \x in {A,...,Z}{%
\expandafter\xdef\csname c\x\endcsname{\noexpand\ensuremath{\noexpand\mathcal{\x}}}
}

\foreach \x in {A,...,Z}{%
\expandafter\xdef\csname b\x\endcsname{\noexpand\ensuremath{\noexpand\mathbb{\x}}}
}

\newcommand{\sm}{\setminus}

\newcommand{\sub}{\subseteq}

\newcommand{\wt}{\widetilde}

\title[Counting spaces of functions on separable compact lines]{Counting spaces of functions\\ on separable compact lines}
\author[M. Korpalski]{Maciej Korpalski}
\author[P. Koszmider]{Piotr Koszmider}
\author[W. Marciszewski]{Witold Marciszewski}
\address{Instytut Matematyczny, Uniwersytet Wroc\l awski, pl.\ Grunwaldzki 2/4, 50-384 Wroc\-\l aw, Poland}
\email{Maciej.Korpalski@math.uni.wroc.pl}
\address{Institute of Mathematics of the Polish Academy of Sciences, ul. \'Sniadeckich 8, 00-656 Warszawa, Poland}
\email{piotr.math@proton.me}
\address{Institute of Mathematics, University of Warsaw, Banacha 2, 02-097 Warszawa, Poland}
\email{W.Marciszewski@mimuw.edu.pl}
\date{}
\subjclass[2020]{Primary 03E35, 03E65, 03E75, 46B03, 46E15, 54F05}

\thanks{The second-named and the third-named authors were partially supported by the NCN (National Science Centre, Poland) research grant no.\ 2020/37/B/ST1/02613.}

\begin{document}

\begin{abstract}
We investigate the following general problem, closely related to the problem of isomorphic classification of Banach spaces $C(K)$ of continuous real-valued functions on a compact space $K$, equipped with the supremum norm:

Let $\mathcal{K}$ be a class of compact spaces. How many isomorphism types of Banach spaces $C(K)$  are there, for $K\in \mathcal{K}$? 

We prove that for any uncountable regular cardinal number $\kappa$, there exist exactly $2^\kappa$ isomorphism types of spaces $C(K)$ for compact spaces of weight $\kappa$.

We show that, for the class $\mathcal{L}_{\omega_1}$ of separable compact linearly ordered spaces of weight $\omega_1$, the answer to the above question depends on additional set-theoretic axioms. In particular, assuming the continuum hypothesis, there are $2^{\omega_1}$ isomorphism types of $C(L)$, for $L\in \mathcal{L}_{\omega_1}$, and assuming a certain axiom proposed by Baumgartner, there is only one type.
\end{abstract}

\maketitle

\section{Introduction} \label{sec:1 Introduction}

One of the most fundamental and natural challenges of mathematics is the problem of classifying objects within a given class. An important part of such a challenge is an attempt to determine how many isomorphism types of objects we have in that class. 

In this paper, we consider the following case of this general problem:

Let $\mathcal{K}$ be a class of compact spaces. How many isomorphism types of Banach spaces $C(K)$ of real-valued continuous functions on $K$, equipped with the supremum norm, do we have for $K\in \mathcal{K}$? 

Recall that for an infinite compact space $K$, the weight of $K$ is equal to the weight and the density of $C(K)$. Hence, if the Banach spaces $C(K)$ and $C(L)$ are isomorphic, then the compact spaces $K$ and $L$ must have equal weights. It is also known that compact spaces $K$ and $L$ with isomorphic function spaces have equal cardinalities \cite{Ce78}.
Therefore, it is natural that in the problem of counting isomorphism types of spaces $C(K)$, we restrict ourselves to the classes of compacta $K$ of fixed weight or fixed cardinality.

The classical results of Miljutin, Bessaga and Pełczyński give us a complete isomorphic classification of spaces $C(K)$ for metrizable compacta $K$, i.e., for compacta of countable weight. 
In particular, Bessaga and Pełczyński (cf.\ \cite{Se71}) proved that there are exactly $\omega_1$ isomorphism types of $C(K)$, for countable (hence metrizable) $K$, and Miljutin (cf.\ \cite{Se71}) showed that, for uncountable metrizable $K$, there is only one isomorphic type of $C(K)$ space.
Summarizing, we have $\omega_1$ isomorphism types of spaces $C(K)$ for compacta $K$ of weight $\omega$ (regardless of the size of the continuum  $2^{\omega}$). 
Recall that we have $2^{\omega}$ topological types of compacta of weight $\omega$ \cite{MS20}.

For the classes of nonmetrizable compacta, the chances of obtaining such complete classification results seem to be very slim, for a simple reason - we usually have too many isomorphism types of spaces $C(K)$. 
It is known that even if we restrict ourselves to some concrete classes of compact spaces of weight $2^\omega$, we still have $2^{2^\omega}$ isomorphism types of $C(K)$.
In particular, it is true for the class of separable scattered compacta of height 3 (cf.\ \cite{CS+20}), or the class of separable, compact linearly ordered spaces (cf.\ \cite{Ma08}).

Note that, for a compact $K$ of weight $\kappa$, the cardinality of isomorphism types of $C(K)$ is bounded by $2^\kappa$ (cf.\ remarks in the last section).

In this paper, we concentrate on the case of compact spaces of weight $\omega_1$. We prove that there are $2^{\omega_1}$ isomorphism types of $C(K)$, for $K$ of such weight. Actually, we prove the following more general result:

\begin{theorem}\label{many}
Suppose that $\kappa$ is an uncountable regular cardinal.
There is a family of cardinality $2^\kappa$ of compact spaces $K$  of weight $\kappa$, such that the corresponding Banach spaces $C(K)$ are pairwise nonisomorphic. 
\end{theorem}

Most of our investigations deal with the class of separable, compact linearly ordered spaces (in short: separable compact lines). 
Banach spaces $C(K)$ associated with such compacta $K$ play an important role in Banach space theory. 
They have provided many interesting examples (see  \cite{Ka99}, \cite{Ka02}, \cite{KMS01}, \cite{MS06}), and have been an object of investigations in numerous papers (see \cite{AK24}, \cite{CT15}, \cite{KK12}, \cite{KP24}, \cite{Ma08}, \cite{Mi20}, \cite{ST23}).

Let $\cL_{\kappa}$ be the class of separable, compact linearly ordered spaces of weight $\kappa$ (for our topological terminology see \cite{En89}).
We show that the cardinality of isomorphism types of $C(L)$, for $L\in \cL_{\omega_1}$, depends on additional set-theoretic axioms. 
In particular, assuming the continuum hypothesis (i.e., $2^{\omega} = \omega_1$), there are $2^{2^\omega}$ isomorphism types of $C(L)$, for $L\in \cL_{\omega_1}$. More precisely, we prove the following:

\begin{theorem} \label{thm: Exactly 2^omega1 spaces}
If $\kappa \le 2^\omega < 2^{\kappa}$, then there exist $2^{\kappa}$ pairwise nonisomorphic spaces $C(L)$ for $L\in \cL_{\kappa}$.
\end{theorem}

On the other hand, assuming a certain axiom (\textsf{BA}) proposed by Baumgartner, there is only one type of $C(L)$ for $L\in \cL_{\omega_1}$:

\begin{theorem}[\textsf{BA}] \label{thm: Under BA there is one C(K_A) space}
If $K$ and $L$ are separable compact lines of weight $\omega_1$, then the Banach spaces $C(K)$ and $C(L)$ are isomorphic.
\end{theorem}

From this theorem we will derive the following:

\begin{corollary}[\textsf{BA}] \label{cor: Under BA C(K_A) is iso with its square}
If $K$ is a separable compact line of weight $\omega_1$, then the Banach space $C(K)$ is isomorphic with the direct sum $C(K)\oplus C(K)$.
\end{corollary}

The above corollary should be compared with a recent result of Kucharski \cite{Ku26} who proved that there exists a separable compact line $K$ of weight $2^\omega$ such that the Banach space $C(K)$ is not isomorphic to $C(K)\oplus C(K)$. In particular, this shows that Corollary \ref{cor: Under BA C(K_A) is iso with its square} requires some additional set-theoretic hypothesis.

Let us recall that another example of a class $\mathcal{K}$ of compact spaces of weight $\omega_1$, such that the cardinality of isomorphism types of $C(K)$, for $K\in \mathcal{K}$, depends on additional set-theoretic axioms, was given in \cite{CS+20} ($\mathcal{K}$ is the class of separable scattered compacta of height 3).
\medskip

The paper is organized as follows:

In Section \ref{prelim} we explain our terminology, recall some definitions and auxiliary facts.

Section \ref{sec_counting_C(K)} contains the proof of Theorem \ref{many} and its prerequisites. 
In Section \ref{sec_top_properties} we present some general topological properties of separable compact lines, in particular counting their homeomorphic types and recalling some handy characterizations.

The main result of Section \ref{sec_properties_C(K_A)}, dedicated to spaces of continuous functions on separable compact lines, is Lemma \ref{lem: C(KA) and C(KA) x C(M) are isomorphic}, which states that $C(K)$ is isomorphic to $C(K) \oplus C(M)$ for any uncountable separable compact line $K$ and any metrizable compact space $M$. 
Sections \ref{sec_counting_under_2^k>2^omega} and \ref{sec_BA} are dedicated to proving Theorems \ref{thm: Exactly 2^omega1 spaces} and \ref{thm: Under BA there is one C(K_A) space}, respectively.

In Section \ref{sec_products}, we turn our attention to spaces of continuous functions on finite products of separable compact lines. Using the results of Michalak \cite{Mi20} and assuming the Baumgartner’s Axiom, we provide a complete isomorphic classification of such $C(K)$ spaces of weight $\omega_1$, see Theorem \ref{thm_isomorphic_classes_of_finite_products}.

The last section, Section \ref{sec_remarks}, is dedicated to some scattered thoughts and remarks about the subject of the paper.

\section{Preliminaries}\label{prelim}

\subsection{Set-theoretic terminology}

We denote by $cf(\alpha)$ the cofinality of the ordinal $\alpha$. 
Let $\kappa$ be a regular cardinal, i.e., $cf(\kappa) = \kappa$. Recall that a set $A\sub \kappa$ is a club set if it is unbounded in $\kappa$ and closed in the order topology. 
A set $S \sub \kappa$ is stationary if it has a nonempty intersection with all club subsets of $\kappa$.
By $\omega_1$ we denote the first uncountable cardinal and usually write $2^\omega$ for the continuum, the cardinality of the set of real numbers. In general, our set-theoretic terminology follows Jech's book \cite{TJ03}.

\subsection{Banach spaces of continuous functions and their duals}
Each compact and locally compact space we consider is Hausdorff. 
For a locally compact space $X$ we denote by $C_0(X)$ the space of real-valued continuous functions on $X$ vanishing at infinity, considered with the supremum norm. 
For a compact space $K$, we write $C(K)$ instead of $C_0(K)$, as it is the Banach space of all real-valued continuous functions on $K$.
It is a classical fact that the density and weight of $C(K)$ is equal to the weight of $K$.
For a set $A \sub K$, we denote the characteristic function of $A$ by $1_A$.

The dual space $C(K)^*$, due to the Riesz representation theorem, can be seen as $M(K)$, the Banach space of signed Radon measures on $K$ with absolute variation as the norm.
There is a similar representation of $C_0(X)^*$ for any locally compact space $X$, where every element of $C_0(X)^*$ can be represented as the integration with respect to a signed Radon measure on the one-point compactification $\alpha X$ of $X$ that vanishes at the added point (see Theorem 18.4.1 of \cite{Se71}). 

In the space $M(K)$, equipped with the weak* topology induced by $C(K)$, there is a linearly dense topological copy of $K$, which is the set of Dirac delta measures $\Delta_K = \{\delta_x : x\in K\}$ (we will sometimes identify $K$ and $\Delta_K$). 
This, in particular, means that if the space $K$ is separable, then $M(K)$ is also separable in the weak* topology.
We write $supp(\mu)$ for the support of the measure $\mu \in M(K)$.

For two Banach spaces $X, Y$ we write $X \simeq Y$ if $X$ and $Y$ are isomorphic.
By $X|Y$ we mean that $Y$ contains a complemented isomorphic copy of $X$, i.e., there is a Banach space $X'$ such that $Y \simeq X \oplus X'$.

For a countable family of Banach spaces $\{X_n : n\in \omega\}$,  their $c_0$-sum is the Banach space consisting of all sequences $(x_n)_{n\in \omega} \in \prod_{n\in \omega} X_n$ such that for every $\epsilon > 0$, the set $\{n\in \omega \colon \|x_n\| > \epsilon\}$ is finite, equipped with the supremum norm $\|(x_n)_{n\in \omega}\| = \sup_{n\in \omega} \|x_n\|$. 

The following fact is standard, see e.g. \cite[Proposition 21.5.8]{Se71}. Note that all hyperplanes of a Banach space are mutually isomorphic, see \cite[Exercise 2.9]{Fa+11}.

\begin{fact} \label{pre: C(K) iso to its hyperplanes}
Let $K$ be a compact space with a nontrivial convergent sequence. Then the Banach space $C(K)$ is isomorphic to its hyperplanes.
\end{fact}
In general, our terminology concerning Banach spaces follows books \cite{Fa+11} and \cite{Se71}.

\subsection{Scattered spaces}

Recall that a topological space $F$ is scattered if every nonempty set $A \sub F$ contains a relative isolated point.

For a scattered space $X$, by the Cantor-Ben\-dixson height (or simply the height) $ht(X)$ of $X$ we mean the least ordinal $\alpha$ such that the Cantor-Bendixson derivative $X^{(\alpha)}$ of the space $X$ is empty. 
The height of a compact scattered space is always a nonlimit ordinal.

Let us make a general remark about the bounded linear functionals on $C_0(X)$ for any locally compact scattered space $X$.
By the Rudin theorem (Theorem 19.7.6 of \cite{Se71}), if $\alpha X$ is scattered, then all Radon measures on the one-point compactification $\alpha X$ are purely atomic. 
From the Riesz representation theorem mentioned above it follows that any bounded linear functional $\phi$ on $C_0(X)$ is defined by
\[\phi(f)=\sum_{n\in \omega}a_n f(x_n),\]
where $\sum_{n\in \omega}|a_n|<\infty$ and $\{x_n: n\in \omega\} \sub X$ are distinct, and $f\in C_0(X)$.

\subsection{Separable compact linearly ordered spaces}\label{subs: sep cpt lines}

Consider an arbitrary closed subset $K$ of the unit interval $I = [0, 1]$ and any subset $A \sub K$. Denote
\[K_A = \big(K \times \{0\}\big) \cup \big( A \times \{1\}\big)\]
and consider this set with the order topology given by the lexicographic order on $[0,1] \times \{0,1\}$.
If $A = (0, 1)$ and $K = I$, then the space $K_A$ is the well-known double arrow space. Some authors use this name for the space
$I_I$; others call the space $I_I$ the split interval. 
For any sets $A, K$ as above, the space $K_A$ is a separable compact line of topological weight $|A|$ (if the set $A$ is infinite), and thus it is nonmetrizable whenever the set $A$ is uncountable.
If $A$ is dense in $K$, then the space $K_A$ is zero-dimensional.
It turns out that spaces of the form $K_A$ exhaust the class of separable compact lines:

\begin{theorem}[Ostaszewski, \cite{Os74}] \label{int:1}
The space $L$ is a separable compact linearly ordered space if and only if $L$ is order isomorphic (hence homeomorphic) to the space $K_A$ for some closed set $K \sub [0, 1]$ and a subset $A \sub K$.
\end{theorem}

Due to Ostaszewski's result, in this paper we may assume that any separable compact line is of the form $K_A$.

\section{Counting isomorphism types of \texorpdfstring{$C(K)$}{C(K)} of a given weight} \label{sec_counting_C(K)}

The main result of this Section is Theorem \ref{many}, in which we count isomorphism types of Banach spaces of continuous functions on compact spaces of regular weight $\kappa$. 
Before we proceed further, let us introduce some notation. For an uncountable cardinal $\kappa$, we denote by
\[E^\kappa_\omega=\{\alpha<\kappa: cf(\alpha)=\omega\}\]
the set of ordinals below $\kappa$ of countable cofinality. Furthermore, if $A, B$ are any sets of ordinals, then $A < B$ means that $\alpha < \beta$ for all $\alpha \in A$, $\beta \in B$. Similarly, if $\alpha$ is an ordinal, then $\alpha < A$ ($A < \alpha$) means that $\{\alpha\} < A$ ($A < \{\alpha\}$).

The following space, known, at least for $\kappa=\omega_1$, as the ladder system space, seems to belong to the topological folklore since the early 1970s. 
Perhaps its best-known early appearances are in \cite{DS79} and \cite{Po79}.

\begin{definition}\label{def-ladder}
Consider an uncountable regular cardinal $\kappa$. 
Let $\cL = \{L_\alpha: \alpha\in E^\kappa_\omega\}$ be a system of ladders, i.e., for $\alpha\in E^\kappa_\omega$ we demand
\begin{itemize}
    \item $L_\alpha \sub \{\beta + 1: \beta < \alpha\}$,
    \item $L_\alpha\cap \beta$ is finite for all $\beta<\alpha$,
    \item $\sup(L_\alpha)=\alpha$.
\end{itemize}
For a subset $S$ of $E^\kappa_\omega$, we define the ladder system space $X_{S}(\cL)=X_S$ on $\kappa$:
\begin{itemize}
    \item The points of $X_S$ are elements of $\kappa$.
    \item The points of $\kappa\setminus S$ are isolated in $X_S$.
    \item The basic neighborhoods of $\alpha\in S$ are declared as all sets of the form
    \[(L_\alpha\setminus F)\cup\{\alpha\},\]
    where $F \sub L_\alpha$ is finite.
\end{itemize}
\end{definition}

The following should be clear:

\begin{proposition} 
Suppose that $\kappa$ is an uncountable regular cardinal.
Any ladder system space on $\kappa$ is a locally compact, scattered space of weight $\kappa$ and height $2$.
\end{proposition}

The one-point compactification of $X_S$ will be denoted $KL_S$. 
It is quite easy to prove that the space $C_0(X_S)$ is isomorphic to $C(KL_S)$.

\begin{lemma}\label{hyper} 
Suppose that $\kappa$ is an uncountable regular cardinal and $X_S$ is a ladder system space on $\kappa$ for some $S \sub E^\kappa_\omega$. 
Then the Banach spaces $C_0(X_S)$ and $C(KL_S)$ are isomorphic.
\end{lemma}

\begin{proof}
By the Stone-Weierstrass theorem, $C_0(X_S)$ is a hyperplane of $C(KL_S)$.
The space $KL_S$ is scattered and infinite, so it admits a nontrivial convergent sequence. 
The rest follows from Fact \ref{pre: C(K) iso to its hyperplanes}.
\end{proof}

\begin{lemma} \label{counting iso types: At most countably many nonzero measures at any point}
Suppose that $\kappa$ is an uncountable regular cardinal and $X_R, X_S$ are two ladder system spaces on $\kappa$ for some sets $R, S \sub E^\kappa_\omega$. If $T : C_0(X_R) \rightarrow C_0(X_S)$ is a bounded linear operator, then for every $\gamma \in \kappa$ we have
\[|\{\alpha \in X_S \colon T^*(\delta_\alpha)(\{\gamma\}) \neq 0\}| < \kappa.\]
\end{lemma}

\begin{proof}  
Assume towards a contradiction that there is $\gamma \in X_R$ such that cardinality the set
\[A = \{\alpha \in X_S \colon T^*(\delta_\alpha)(\{\gamma\}) \neq 0\}\]
equals $\kappa$. 
By passing to a subset of $A$ of cardinality $\kappa$, we may assume that there is $\eps>0$ such that $|T^*(\delta_\alpha)(\{\gamma\})|>\eps$ for all $\alpha\in A$. 
If $\gamma\in X_R$ is isolated in $X_R$, then $1_{\{\gamma\}}$ is continuous and then
\[|T(1_{\{\gamma\}})(\alpha)|= |T^*(\delta_\alpha)(\{\gamma\})|>\eps\]
for all $\alpha\in A$, which is impossible since $A$ is noncompact as a subset of $\kappa$ of cardinality $\kappa$, but continuous functions in $C_0(X_S)$ must vanish at infinity. 

Otherwise, $\gamma$ is not isolated in $X_R$. 
Let $L_\gamma$ be the ladder at $\gamma$.  
By passing to a subset of $A$ of cardinality $\kappa$, using the regularity of $\kappa$ and the regularity of the measure $T^*(\delta_\alpha)$, we may assume that there is a cofinite subset $L$ of $L_\gamma$ such that $|T^*(\delta_\alpha)|(L)<\eps/2$ for every $\alpha\in A$. 
Now $1_{L\cup\{\gamma\}}$ is continuous on $X_R$ and
\[|T(1_{L\cup\{\gamma\}})(\alpha)|= |T^*(\delta_\alpha)(L\cup\gamma)|>\eps-\eps/2=\eps/2\]
for all $\alpha\in A$, which is again impossible, as continuous functions in $C_0(X_S)$ must vanish at infinity. 
\end{proof}

The following lemma will be the main tool for proving Theorem \ref{many}.

\begin{lemma}\label{notsurjective}
Suppose that $\kappa$ is an uncountable regular cardinal and that $R, S \sub E_\omega^\kappa$ are two sets such that $S\setminus R$ is stationary in $\kappa$. 
Then there is no bounded linear operator from $C_0(X_R)$ into $C_0(X_S)$ with a dense range. 

Consequently, there is no bounded linear operator $C(KL_R)$ into $C(KL_S)$ with a dense range.
\end{lemma}

\begin{proof}
Suppose that $T: C_0(X_R)\rightarrow C_0(X_S)$ is a bounded linear operator. 
We will show that it does not have a dense range.
Consider the adjoint operator $T^*: C_0(X_S)^*\rightarrow C_0(X_R)^*$. 
Let $F, G: \kappa\rightarrow \kappa$ be nondecreasing functions such that 
\[supp(T^*(\delta_{\alpha})) \sub F(\alpha),\]
\[\bigcup\{supp(T^*(\delta_{\beta})): G(\alpha)<\beta<\kappa\}\cap \alpha=\emptyset\]
for every $\alpha<\kappa$. 
The existence of $F$ follows from the fact that the supports of Radon measures on $\alpha X_R$ that vanish on the one point added in the compactification are countable and such sets must be bounded in $\kappa$ since it is an uncountable regular cardinal.

The proof of the existence of $G$ uses the fact that $T$ is bounded: if the value $G(\alpha)$ could not be found, there would exist $A \sub \kappa$ of cardinality $\kappa$ and a fixed $\gamma<\alpha$ such that $T^*(\delta_\beta)(\{\gamma\})\not=0$ for all $\beta\in A$, which contradicts Lemma \ref{counting iso types: At most countably many nonzero measures at any point}. 

Using the standard closure argument one can easily verify that, for any function $H: \kappa\rightarrow \kappa$, the set $\{\alpha < \kappa \colon H[\alpha]\sub \alpha\}$ is a club set in $\kappa$. Since the intersection of two club sets in $\kappa$ is a club set in $\kappa$ we obtain a club set $C \sub\kappa$ which is closed under $F$ and $G$, that is, for every $\alpha\in C$ we have $F[\alpha], G[\alpha] \sub \alpha$.

Fix $\alpha\in C\cap(S\setminus R)$. By recursion, construct a strictly increasing sequence $\beta_n'\in L_\alpha$ such that $F(\beta_n'), G(\beta_n')<\beta_{n+1}'$. 
This can be done since $C$ is closed under $F$ and $G$, and $L_\alpha$ is cofinal in $\alpha$. 
The definitions of $F$ and $G$ imply that
\[supp(T^*(\delta_{\beta_{n+1}'})) \sub [\beta_n', \beta_{n+2}')\]
for every $n\in \omega$.

Now define $\beta_n=\beta_{2n}'$ and note that the supports of $T^*(\delta_{\beta_{n}})$ satisfy
\[supp(T^*(\delta_{\beta_{n}}))<\beta_{2n+1}'<supp(T^*(\delta_{\beta_{n+1}}))<\alpha.\]
We can see that $(\beta_n)_{n\in\omega}$ converges to $\alpha$ in $X_S$ as it is a strictly increasing sequence included in $L_\alpha$. 
Hence $(\delta_{\beta_n})_{n\in\omega}$ weakly$^*$ converges to $\delta_\alpha$ in $C_0(X_S)^*$.
The weak$^*$ continuity of $T^*$ yields the weak$^*$ convergence of $(T^*(\delta_{\beta_n}))_{n\in\omega}$ to $T^*(\delta_\alpha)$ in $C_0(X_R)^*$.

We will show that in fact $(T^*(\delta_{\beta_n}))_{n\in\omega}$ weakly$^*$ converges to $0$ in $C_0(X_R)$, which will imply $T^*(\delta_\alpha)=0$ showing that $T^*$ is not injective and hence $T$ cannot have a dense range.

Indeed, when $f\in C_0(X_R)$, then for every $\eps>0$ there is $\beta<\alpha$ such that all values of $f|[\beta, \alpha)$ are in the interval $(-\eps, \eps)$. 
Otherwise, there would be a strictly increasing and cofinal sequence $(\alpha_n)_{n\in \omega}$ of points in $\alpha$ satisfying $|f(\alpha_n)| \geq \eps$.
As the point $\alpha$ is isolated in $X_R$ and the basic neighborhood of other points in $X_R$ have bounded intersections with $\alpha$, the sequence $(\alpha_n)_{n\in \omega}$ is divergent to infinity, which is a contradiction.

It follows from the above and the fact that $T^*(\delta_{\beta_n})$ form a bounded sequence of measures that the sequence $T^*(\delta_{\beta_n})(f)$ converges to $0$ for each $f\in C_0(X_R)$. 
Thus, we have that the sequence $(T^*(\delta_{\beta_n}))_{n\in\omega}$ weakly$^*$ converges to $0$ in $C_0(X_S)^*$ and $T^*(\delta_\alpha) = 0$ as needed to conclude that $T$ does not have dense range.
\end{proof}

We can now proceed to the proof of Theorem \ref{many}.

\begin{proof}[Proof of Theorem \ref{many}]
Fix $\kappa$ as in the statement. By a theorem of Solovay (Theorem 8.10 of \cite{TJ03}
also cf. Lemma 8.8 of \cite{TJ03}) there is a pairwise disjoint family $\{S_\xi: \xi<\kappa\}$
of stationary subsets of $E_\omega^\kappa$. For every $A\subseteq\kappa$ consider
\[S(A)=\bigcup_{\xi\in A} S_\xi.\]
Note that $S(A)\triangle S(B)$ is stationary for any two distinct $A, B \sub \kappa$.
So, by Lemmas \ref{notsurjective} and \ref{hyper} the family $\{C(KL_{S(A)}) \colon A \sub \kappa\}$ is the desired family of Banach spaces.
\end{proof}

\section{Topological properties of separable compact lines} \label{sec_top_properties}

In this section we use our convention that we identify separable compact lines with the spaces of the form $K_A$, where $K\subseteq [0,1]$ and $A\subseteq K$, see subsection \ref{subs: sep cpt lines}.

First, we will show that for any cardinal number $\kappa$ satisfying $\omega\le \kappa\le 2^\omega$, there are $2^{\kappa}$ pairwise nonhomeomorphic separable compact lines $K_A$ of weight $\kappa$ (recall that, for infinite $A$, the weight of $K_A$ is equal to $|A|$). 

We start with the following observation of van Douwen \cite[Section 10]{vD84}. 

\begin{proposition}
Let $K_A, L_B$ be separable compact lines that are homeomorphic. Then there exist countable sets $C \sub K$ and $D \sub L$ such that $K \sm (A \cup C)$ and $L \sm (B \cup D)$ are homeomorphic.
\end{proposition}

The above proposition together with a standard fact that each subset of $I$ can be homeomorphic to at most continuum many subsets of $I$, and some routine cardinal calculations, easily implies that for a fixed separable compact line  $K_A$ we have at most $2^\omega$ many separable compact lines $L_B$ homeomorphic to $K_A$ (cf.\ \cite[Section 10]{vD84}, \cite[Section 2]{Ma08}). This immediately allows us to obtain the following variation of van Douwen's result that there are $2^{2^\omega}$ separable compact lines. 

\begin{corollary} \label{cor: At least 2^omega1 spaces}
If a cardinal number $\kappa\le 2^\omega$ satisfies $2^{\kappa} > 2^\omega$, then there are $2^{\kappa}$ pairwise nonhomeomorphic separable compact lines $K_A$ of weight $\kappa$.
\end{corollary}

We will show that the above result also holds true when $2^{\kappa} = 2^\omega$.

\begin{theorem}\label{thm: At least 2^omega1 spaces}
For any infinite cardinal number $\kappa\le 2^\omega$, there are $2^{\kappa}$ pairwise nonhomeomorphic separable compact lines $K_A$ of weight $\kappa$.
\end{theorem}

\begin{proof}
By Corollary \ref{cor: At least 2^omega1 spaces} it is enough to consider the case when $2^{\kappa} = 2^\omega$, and construct a family $\mathcal{K}$ of size $2^\omega$ consisting of pairwise nonhomeomorphic spaces $K_A$ of weight $\kappa$. 

By a classical result of Mazurkiewicz and Sierpiński \cite{MS20} there is a family $\cL$ of size $2^\omega$ consisting of pairwise nonhomeomorphic compact subsets of the interval $[0,1/3]$. If $\kappa = \omega$, then we are done. Otherwise, we take a  subset $A$ of $(2/3,1)$ such that $|A\cap P| = \kappa$, for any nonempty open subinterval $P$ of $(2/3,1)$, and we
put $\mathcal{K} = \{L\cup [2/3,1]_A: L\in \cL\}$. The family $\mathcal{K}$ has the required properties, since each point of $[2/3,1]_A$ has only nonmetrizable neighborhoods, hence $L\cup [2/3,1]_A$ and $L'\cup [2/3,1]_A \in \mathcal{K}$ are homeomorphic if and only if the spaces $L$ and $L'$ are homeomorphic.
\end{proof}

\begin{proposition} \label{prop_properties_K_A}
Each separable compact line is hereditarily separable, hereditarily Lindel\"of, and Fr\'echet.
\end{proposition}

\begin{proof}
Every separable compact line has to be Fr\'echet, since it is a first countable space.
Each separable compact line is of the form $K_A$, so it is a continuous image of the double arrow space $I_{(0, 1)}$ and all these properties are preserved by continuous images.
The double arrow space is hereditarily separable, hereditarily Lindel\"of, as it is a union of two disjoint copies of the Sorgenfrey line, which is such.
\end{proof}

Observe that the above proposition implies, in particular, that the class of separable compact lines is closed under taking closed subsets.
\medskip

For an ordinal $\alpha$, $X^{(\alpha)}$ is the $\alpha$th Cantor-Bendixson derivative of the space $X$. 
The following result follows quite easily from Proposition \ref{prop_properties_K_A}.

\begin{corollary}\label{cor_properties_K^(alpha)}
Let $L$ be a separable compact line and $\alpha = \min\{\beta: L^{(\beta+1)} = L^{(\beta)}\}$. Then $\alpha < \omega_1$ and $|L\setminus  L^{(\alpha)}| \le \omega$.
\end{corollary}

If a separable compact line is perfect (without isolated points), we can modify it by a homeomorphism to a more convenient form.

\begin{proposition}\label{prop_no_isolated_points}
Let $L$ be a separable compact line. If $L$ has no isolated points, then there exists $B\subseteq (0,1)$ such that $L$ is order isomorphic (hence homeomorphic) to $I_B$.
\end{proposition}

\begin{proof}
Without loss of generality, we can assume that $L = K_A$, where $A \sub K \sub I$, and $K$ is a closed set such that $\inf K = 0$ and $\sup K = 1$.

Enumerate the family of all connectedness components of $I \sm K$ as $\{(a_n, b_n) : n \in N\}$, where $N = \omega$ or $N\in \omega$. Observe that, for each $n \in N$, the set $\{a_n,b_n\}$ is disjoint from $A$, since $K_A$ has no isolated points.

Now consider an equivalence relation on $K$ defined by $a_n \! \sim b_n$ for all $n\in N$ and let $q : K \to K/\!\!\sim$ be the quotient map.
Notice that the relation $\!\sim$ is a closed relation (see \cite{En89}), as every increasing or decreasing sequence of $a_n$'s will have the same limit as a sequence of $b_n$'s with the same indicies. 
One can easily check that the resulting quotient space $K /\!\! \sim$ is a linearly ordered connected metrizable compact space, hence it is order isomorphic to the unit interval (cf.\ \cite[6.3.2]{En89}). 
Denote by $h : K/\!\!\sim \ \to I$ the order isomorphism satisfying $h(q(0)) = 0$.
One can verify that $K_A$ can be identified with $I_B$ where $B = h[q[A \cup \{a_n : n \in N\}]]$.
\end{proof}

\begin{proposition} \label{prop_zero_dim_no_isolated_points}
Let $L$ be a separable zero-dimensional compact line. Let $C \sub I$ be the ternary Cantor set. If $L$ has no isolated points, then there exists $B\subset C$ such that $L$ is order isomorphic (hence homeomorphic) to $C_B$.
\end{proposition}

\begin{proof}
By Proposition \ref{prop_no_isolated_points}, we can assume that $L = I_A$ for some $A \sub I$. As $L$ is zero-dimensional, $A$ is dense in $I$. 
We can find a copy of the dyadic numbers $Q = \{a_s : s \in 2^{<\omega}\}$ dense in $A$, so that $a_s < a_{s'}$ (in $A$) if and only if $s < s'$ (in $2^{<\omega}$ with the lexicographical order). 
Let $\{(q_s^0, q_s^1) \sub I : s\in 2^{<\omega}\}$ be the usual, order preserving enumeration of the intervals removed during the construction of the Cantor set $C \sub I$. 

Define a function $h : Q \times \{0, 1\} \to C$ by 
\[h(a_s, b) = q_s^b\]
for $s \in 2^{<\omega}, b\in \{0, 1\}$. Since $h$ is order-preserving (with respect to the lexicographic order on $Q \times \{0, 1\}$), $Q$ is dense in $I$, and $\{q_s^0, q_s^1 : s\in 2^{<\omega}\}$ is dense in $C$, we can extend $h$ to an order isomorphism $\wt{h} : I_A \to C_B$, where $B\times \{1\} = \wt{h}\big[(A \sm Q\big) \times \{1\}]$. 
\end{proof}

\section{General properties of spaces \texorpdfstring{$C(K_A)$}{C(K A)}} \label{sec_properties_C(K_A)}

For a closed subset $F$ of a compact space $K$, we denote by $C_0(K\|F)$ the subspace $\{f \in C(K) : f|F \equiv 0\}$ of $C(K)$. 

\begin{proposition}
Let $F$ be a closed subset of a Fr\'echet compact space $K$, such that $K\setminus F$ is infinite. Then $C_0(K\|F)\simeq C(K/F)$.
\end{proposition}
    
\begin{proof}
The space $C_0(K\|F)$ is isomorphic to a hyperplane of $C(K/F)$. 
From our hypothesis, it follows that $K/F$ contains a nontrivial convergent sequence, thus the statement follows from Fact \ref{pre: C(K) iso to its hyperplanes}.
\end{proof}

\begin{corollary}\label{cor_C(K,F)}
Let $F$ be a closed subset of a separable compact line $K$.
Then there is a compact space $L$ such that $C_0(K\|F)\simeq C(L)$.
\end{corollary}

\begin{proof}
If $K\setminus F$ is infinite, then we can apply the above proposition. Otherwise, the space $C_0(K\|F)$ is finite-dimensional, and we can take as $L$ a discrete finite space of appropriate size.
\end{proof}

The next result follows from a theorem of Heath and Lutzer \cite[Theorem 2.4]{HL74} which provides the existence of an extension operator $T : C(L) \to C(K)$ for compact lines $L\sub K$ (see \cite[Lemma 4.3]{Ma08} for a self-contained proof).

\begin{lemma} [Marciszewski, \cite{Ma08}*{Lemma 4.3}] \label{lem: decomposition through a closed subspace}
Let $K_A$ be a separable compact line. For each nonempty closed subset $F$ of $K_A$, the space $C(K_A)$ is isomorphic to $C(F)\oplus C_0(K_A\|F)$.
\end{lemma}

The following result is a more general version of \cite[Lemma 4.6]{Ma08}. Note that it is true in ZFC, without any additional set-theoretic hypothesis. The proof of this lemma follows closely the argument in \cite{Ma08}, we present it here for the convenience of readers. 

\begin{lemma} \label{lem: C(KA) and C(KA) x C(M) are isomorphic}
For every uncountable separable compact line $K_A$ and every nonempty metrizable compact space $M$ the spaces $C(K_A)$ and $C(K_A) \oplus C(M)$ are isomorphic.
\end{lemma}

\begin{proof}
First, let us show that $C(2^\omega)|C(K_A)$.

As $K$ is an uncountable closed subset of the unit interval, we can find a copy of the Cantor set $C \sub K$. 
Put $B = C \cap A$ and see that $C_B$ is a closed subspace of $K_A$. By Lemma \ref{lem: decomposition through a closed subspace}, we have $C(C_B)|C(K_A)$. 
Now, it is sufficient to show that $C(2^\omega)|C(C_B)$.

Write $p_1 : C_B \to C$ for the projection.
Fix a homeomorphism $h : C \to C_1 \times C_2$, where both $C_1, C_2$ are copies of $2^\omega$ and by $\pi_1, \pi_2$ denote the projections from $ C_1 \times C_2$ onto $C_1, C_2$, respectively.

Our goal is to construct a regular averaging operator for the map 
\[\phi = \pi_1 \circ h \circ p_1 : C_B \to C_1,\]
i.e., a positive linear operator $T : C(C_B) \to C(C_1)$ satisfying $T (1_{C_B}) = 1_{C_1}$ and $T(g \circ \phi) = g$ for every $g \in C(C_1)$.
The existence of such an operator implies that $C(C_1)|C(L)$, see \cite[Proposition 8.2]{Pe68}.

The set $C_2$ is a copy of $2^\omega$, so there is a standard product measure on $C_2$, denote it by $\mu$. Define the operator $T : C(C_B) \to C(C_1)$ by the formula
\[T(f)(x) = \int_{C_2} f(h^{-1}(x, y), 0) \diff \mu(y)\]
for $f \in C(C_B)$ and $x \in C_1$. 
It is a well-known fact that, for every function $f \in C(C_B)$, the function $f_0(t) = f(t, 0)$ has at most countably many points of discontinuity (see e.g.\ \cite[Proof of Thm.\ 3.7]{Ma08}) and is therefore measurable. 
It follows that $T(f)(x)$ is well defined for every $x\in C_1$, it is also clear that $T$ is positive and  $T (1_{C_B}) = 1_{C_1}$.
Additionally, for every $g \in C(C_1)$, $x\in C_1$ and $y \in C_2$ we have $g \circ \phi(h^{-1}(x, y), 0) = g(x)$, hence $T(g \circ \phi) = g$. 

It remains to check that $T(f)$ is continuous on $C_1$ for every $f \in C(C_B)$. Consider any convergent sequence $x_n \to x$ in $C_1$. Then for every $y \in C_2$ we have $h^{-1}(x_n, y) \to h^{-1}(x, y)$, as $h$ is a homeomorphism.
Consider again the function $f_0(t) = f(t, 0)$ and let $P$ be the set of points of discontinuity of $f_0$ ($P$ is countable).
Denote $f_n(y) = f_0(h^{-1}(x_n, y))$ and $\widetilde{f}(y) = f_0(h^{-1}(x, y))$.
If $R = \pi_2 \circ h(P)$, then for every point $y \in C_2 \sm R$ we have $f_n(y) \to \widetilde{f}(y)$, thus the sequence $(f_n)_{n\in\omega}$ is $\mu$-a.e. convergent to $\widetilde{f}$ on $C_2$ (note that $\mu(R) = 0$, as $R$ is countable and the measure $\mu$ is nonatomic).
As all functions $|f_n|$ are bounded by the norm of $f$, from the Lebesgue dominated convergence theorem it follows that $T(f)(x_n) \to T(f)(x)$.

Now consider any nonempty metrizable compact space $M$.
It is a well-known fact that $C(M) \oplus C(2^\omega) \simeq C(2^\omega)$ (e.g. by Miljutin's theorem).
Since $C(2^\omega)$ is a factor of $C(K_A)$, it follows that $C(M) \oplus C(K_A) \simeq C(K_A)$, as
\[C(M) \oplus C(K_A) \simeq C(M) \oplus C(2^\omega) \oplus E \simeq C(2^\omega) \oplus E \simeq C(K_A),\]
for $E$ such that $C(K_A) \simeq C(2^\omega) \oplus E$.
\end{proof}

\begin{corollary}\label{cor: C(L) and C(L^(alpha)}
Let $L$ be an uncountable separable compact line and $\alpha = \min\{\beta: L^{(\beta+1)} = L^{(\beta)}\}$. Then the spaces $C(L)$ and $C(L^{(\alpha)})$ are isomorphic.
\end{corollary}

\begin{proof}
By Corollary \ref{cor_properties_K^(alpha)} we have $|L\setminus  L^{(\alpha)}| \le \omega$. 
Let $K$ be a compact space given by Corollary \ref{cor_C(K,F)} such that $C_0(L\|L^{(\alpha)})\simeq C(K)$. 
Then either $K= L/L^{(\alpha)}$ or $K$ is finite, in both cases the space $K$ is metrizable since $L\setminus  L^{(\alpha)}$ is countable. 
Therefore, we can use Lemmas \ref{lem: decomposition through a closed subspace} and \ref{lem: C(KA) and C(KA) x C(M) are isomorphic} to conclude that
\[C(L)\simeq C(L^{(\alpha)})\oplus C_0(L\|L^{(\alpha)}) \simeq  C(L^{(\alpha)})\oplus C(K) \simeq C(L^{(\alpha)}).\qedhere\]
\end{proof}

The next theorem follows immediately from the corollary above and Proposition \ref{prop_no_isolated_points}.

\begin{theorem}\label{thm: C(K_A) sim C(I_B)}
Let $L$ be an uncountable separable compact line. Then there exists $A\subseteq (0,1)$ such that $C(L) \simeq C(I_A)$.
\end{theorem}

\section{Counting spaces \texorpdfstring{$C(K_A)$}{C(K A)} of weight \texorpdfstring{$\kappa$}{kappa} when \texorpdfstring{$2^{\kappa} > 2^\omega$}{pow(2, omega1) > pow(2, omega)}} \label{sec_counting_under_2^k>2^omega}

It is a well-known fact that, consistently (e.g.\ under \textsf{CH}), $2^{\omega_1}$ can be greater than $2^\omega$. 
The purpose of this Section is the proof of Theorem \ref{thm: Exactly 2^omega1 spaces}, which we will derive from a result obtained in \cite{CS+20}.

First, we count the number of potential isomorphisms between spaces $C(K_A)$. Let us recall the following, very general theorem due to Cabello-S\'{a}nchez, Castillo, Marciszewski, Plebanek and Salguero-Alarcón. 

\begin{theorem} \cite[Theorem 5.1]{CS+20} \label{thm: At most c isomorphic spaces}
Let $\cK$ be a family of compact spaces such that
\begin{itemize}
    \item $K$ is separable and $|M(K)| = 2^\omega$ for every $K \in \cK$;
    \item For every pair of distinct $K, L \in \cK$ one has $C(K) \simeq C(L)$ and $K, L$ are not homeomorphic.
\end{itemize}
Then $\cK$ is of cardinality at most $2^\omega$.
\end{theorem}

\begin{proof}[Proof of Theorem \ref{thm: Exactly 2^omega1 spaces}]
It is a known fact that for any separable compact line $K_A$ we have $|M(K_A)| = 2^\omega$, see \cite{KMS01}. 

Due to Theorem \ref{thm: At least 2^omega1 spaces}, there exists a family $\mathcal{K}$ of size $2^{\kappa}$ consisting of pairwise nonhomeomorphic separable compact lines $K_A$ of weight $\kappa$. For a fixed $K_A\in \mathcal{K}$, by Theorem \ref{thm: At most c isomorphic spaces}, we have at most $2^\omega$ spaces $L_B \in \mathcal{K}$ such that  $C(K_A) \simeq C(L_B)$. Now, the statement of Theorem \ref{thm: Exactly 2^omega1 spaces}  easily follows from the hypothesis that $2^\kappa > 2^\omega$.
\end{proof}

\section{Counting spaces \texorpdfstring{$C(K_A)$}{C(K A)} under Baumgartner's Axiom} \label{sec_BA}

In this Section we will prove Theorem \ref{thm: Under BA there is one C(K_A) space}, i.e., we will show that consistently we can have a situation completely opposite to the one described by Theorem \ref{thm: Exactly 2^omega1 spaces}. 
The proof is presented at the end of this section, after all the necessary lemmas.

Note that by Theorem \ref{thm: At least 2^omega1 spaces} there exist (in \textsf{ZFC}) $2^{\omega_1}$ pairwise nonhomeomorphic separable compact lines of weight $\omega_1$.

Classical Cantor's isomorphism theorem states that all countable dense subsets of $\bR$ are order isomorphic. The same cannot be said about uncountable subsets, but, after adding a very natural condition, a similar statement can be proved to be consistent with \textsf{ZFC}. For formulating this statement we need the following definition.

\begin{definition}
Let $\kappa$ be an infinite cardinal. We say that a subset $A$ of a topological space $X$ is $\kappa$-dense in $X$, if for any nonempty open subset $U$ of $X$ we have $|U \cap A| = \kappa$.
\end{definition}

We will mainly use this definition in the case when $X$ is a subspace of the real line $\bR$, so we will consider only cardinals $\kappa \le 2^\omega$. 
Obviously, a subset $A$ of $\bR$ is $\kappa$-dense in $\bR$ if, for every nonempty open interval $I \sub \bR$, we have $|I \cap A| = \kappa$. We refer to the following statement about $\kappa$-dense sets in $\bR$ as Baumgartner's Axiom.\medskip

\begin{center}
\textsf{BA}($\kappa$): Every two sets of reals that are $\kappa$-dense in $\bR$ are order isomorphic.
\end{center}\medskip

This axiom was introduced as a natural consequence of a classical theorem of Baumgartner \cite{Ba73}, who proved that consistently all $\omega_1$-dense sets of reals are order isomorphic, i.e., using the above language he showed the following (later on, we will denote $\textsf{BA}(\omega_1)$ simply as $\textsf{BA}$).

\begin{theorem}[Baumgartner] \label{thm: All omega1-dense sets are isomorphic}
\normalfont{\textsf{BA}}$(\omega_1)$ is consistent with \textsf{ZFC}.
\end{theorem}

On the other hand, it should be clear that $\textsf{BA}(\mathfrak{c})$ is false, so consistently $\textsf{BA}(\omega_1)$ is false (e.g.\ under \textsf{CH}). 
Todorčević in \cite{To89} even showed that $\textsf{BA}(\mathfrak{b})$ is also false. 
Theorem \ref{thm: All omega1-dense sets are isomorphic} can be proven by forcing (as in the original proof), or using \textsf{PFA} as in \cite[Theorem 6.9, p. 945]{Ba84}.

The consistency of $\textsf{BA}(\omega_2)$ is an important open problem in set theory, see \cite{MT17}.

Using Dedekind cuts, one can easily obtain the following lemma.

\begin{lemma} \label{lem: Extending order isomorphism} Every order isomorphism $f : A \to B$ between $\omega_1$-dense sets $A, B \sub \bR$ extends to an order isomorphism $\wt{f} : \bR \to \bR$.
\end{lemma}

We can easily transfer such results from the real line to the unit interval $(0, 1)$, since these two spaces are order isomorphic. 
Therefore, we can use Lemma \ref{lem: Extending order isomorphism} to prove some structural results about separable compact lines.

\begin{corollary} \label{cor: A simeq B, then I_A simeq I_B}
If $\omega_1$-dense sets $A, B \sub (0,1)$ are order isomorphic, then the spaces $I_A, I_B$ are homeomorphic.
\end{corollary}

\begin{proof}
Suppose that there exists an order isomorphism $f : A \to B$. By Lemma \ref{lem: Extending order isomorphism} and the remark following it, we extend it to an order isomorphism $\wt{f} : (0,1) \to (0,1)$, which in turn, can be extended in an obvious way to an order isomorphism $\hat{f} : [0,1] \to [0,1]$. Consider the map $\phi: I_A \to I_B$ defined by the formula $\phi(x, i) = (\hat{f}(x), i)$ for $(x, i) \in I_A$. Clearly, $\phi$ is an order isomorphism between $I_A$ and $I_B$, equipped with the lexicographic orders, hence it is a homeomorphism.
\end{proof}

As a consequence, it should be clear that under $\textsf{BA}$ we obtain the following.

\begin{corollary}[\textsf{BA}] \label{cor: IA, IB are homeomorphic}
If sets $A, B\sub (0, 1)$ are $\omega_1$-dense, then the spaces $I_A$ and $I_B$ are homeomorphic.
\end{corollary}

We need one more lemma which does not require any additional set-theoretic axioms.

\begin{lemma}\label{lem: reduction to omega_1-dense sets}
For any separable compact line $K$ of weight $\omega_1$, the space $C(K)$ is isomorphic to a space $C(I_B)$ for some $\omega_1$-dense set $B \sub (0, 1)$.
\end{lemma}

\begin{proof}
Using Theorem \ref{thm: C(K_A) sim C(I_B)} we can assume that $K = I_A$, where $A$ is a subset of $(0,1)$ of size $\omega_1$. Let 
\[U = \bigcup \{(a, b) \sub (0,1): a, b \in \bQ \mbox{ and } |(a, b) \cap A| \le  \omega\}.\]
Clearly, $|U\cap A| \le \omega$ and $A\sm U$ is $\omega_1$-dense in a compact set $M = I \sm U$. We will consider a separable compact line $L = M_{A\sm U}$. 
First, note that $L$ is a closed subset of $K$, hence, by Lemma \ref{lem: decomposition through a closed subspace} we have
\begin{eqnarray}\label{eq: decomp}
C(K) \simeq C(L) \oplus C_0(K\|L)\,.
\end{eqnarray}
Next, let us verify that the space $C_0(K\|L) = C_0(I_{A}\| M_{A\sm U})$ is separable. 

The space $C(I_{A\cap U})$ is separable, since $A\cap U$ is countable; therefore its subspace $C_0(I_{A\cap U}\|M \times \{0\})$  is also separable. Let $\varphi: I_A \to I_{A\cap U}$ be a natural continuous surjection defined by the formula:
\[\varphi((x,i)) = \begin{cases} (x,0)&\quad \mbox{for }  (x,i) \in (A\sm U) \times \{1\}\,,\\
 (x,i)&  \quad \mbox{for }  (x,i) \in I_A \sm [(A\sm U) \times \{1\}]\,.
\end{cases}\]
One can easily check that the isometric embedding $T_\varphi : C(I_{A\cap U}) \to C(I_A)$ generated by $\varphi$, i.e., an operator sending  each $f\in C(I_{A\cap U})$ to the composition $f\circ \varphi$, maps the space $C_0(I_{A\cap U}\|M \times \{0\})$ onto $C_0(I_{A}\| M_{A\sm U})$, hence the latter space is also separable.

By Corollary \ref{cor_C(K,F)} the space $C_0(K\|L) = C_0(I_{A}\| M_{A\sm U})$ is isomorphic to the space $C(N)$ for some compact space $N$. Since $C_0(K\|L)$ is separable, $N$ must be metrizable, so we can use Lemma \ref{lem: C(KA) and C(KA) x C(M) are isomorphic} and property (1) to conclude that $C(K)$ and $C(L)$ are isomorphic.

Let $\alpha = \min\{\beta: L^{(\beta+1)} = L^{(\beta)}\}$. By Corollary \ref{cor_properties_K^(alpha)} we have $|L\setminus  L^{(\alpha)}| \le \omega$. The space $L^{(\alpha)}$ has no isolated points, therefore, by Proposition \ref{prop_no_isolated_points} it is order isomorphic (hence homeomorphic) to a space $I_B$ for some $B\subseteq (0,1)$.

The set $A\sm U$ is $\omega_1$-dense in $M$, hence each nonempty open interval in $L =  M_{A\sm U}$ (an interval with respect to lexicographic order in $L$) contains $\omega_1$ many gaps, i.e., empty intervals of the form $\big((a,0),(a,1)\big)$, where $a\in A\sm U$. 
Since $|L\setminus  L^{(\alpha)}| \le \omega$, the spaces $L^{(\alpha)}$ and $I_B$ have the same property, that is, every nonempty order open interval contains $\omega_1$ gaps. 
This means that the set $B$ is 
$\omega_1$-dense in $I$. By Corollary \ref{cor: C(L) and C(L^(alpha)} the spaces $C(L)$ and $C(L^{(\alpha)})$ are isomorphic, hence we obtain the desired conclusion
\[C(K) \simeq C(L) \simeq C(L^{(\alpha)}) \simeq C(I_B)\,.\qedhere\]
\end{proof}

We can now complete the proof of the main result of this section.

\begin{proof}[Proof of Theorem \ref{thm: Under BA there is one C(K_A) space}]
 The theorem follows immediately from Corollary \ref{cor: IA, IB are homeomorphic} and Lemma \ref{lem: reduction to omega_1-dense sets}.
\end{proof}

Corollary \ref{cor: Under BA C(K_A) is iso with its square} easily follows from Theorem \ref{thm: Under BA there is one C(K_A) space}:

\begin{proof}[Proof of Corollary \ref{cor: Under BA C(K_A) is iso with its square}]
Assume $\textsf{BA}$, and let $K$ be an arbitrary separable compact line of weight $\omega_1$. The direct sum $C(K)\oplus C(K)$ can be identified with the space $C(K\sqcup K)$, where $K\sqcup K$ is a disjoint union of two copies of the space $K$. Since $K\sqcup K$ is homeomorphic to a separable compact line of weight $\omega_1$, Theorem \ref{thm: Under BA there is one C(K_A) space} implies that $C(K)$ and $C(K)\oplus C(K)$ are isomorphic.
\end{proof}

\section{Function spaces on finite products of separable compact lines} \label{sec_products}

In this section, using Theorem \ref{thm: Under BA there is one C(K_A) space}, the results of Michalak from \cite{Mi20}, and assuming Baumgartner's Axiom, we consistently obtain a complete isomorphic classification of spaces $C(K)$, where $K$ is a nonmetrizable finite product of separable compact lines of weight $\le \omega_1$.

For the entirety of this Section, fix some set $\mathbf{A} \sub (0,1)$ of cardinality $\omega_1$.
Define a family of compact spaces
\begin{align*}
\cK_\mathbf{A} = \{(I_\mathbf{A})^n : n \geq 1\} &\cup \{(I_\mathbf{A})^n \times [0,1] : n \geq 1\}\\
&\cup \{(I_\mathbf{A})^n\times (\omega^{\omega^\alpha}+1) : n \geq 1, \omega_1 > \alpha > 0\}.
\end{align*}

The goal of this Section is to prove that the class $\cK_\mathbf{A}$ (regardless of which set $\mathbf{A} \sub (0, 1)$ of cardinality $\omega_1$ is chosen) exhausts all isomorphism types of $C(K)$, where $K$ is any nonmetrizable finite product of separable compact lines of weight $\le \omega_1$.
First, let us note that the spaces $C(K)$ for $K \in \cK_\mathbf{A}$ are pairwise nonisomorphic due to the following result of Michalak.

\begin{theorem}[\cite{Mi20}*{Corollary 5.3}] \label{thm_different_C(K)_are_noniso}
Let $n, k \in \bN$. Let $K_1, \dotsc, K_n, L_1, \dotsc, L_k$ be nonmetrizable separable compact lines.
\begin{enumerate}[(a)]
    \item If $M_1, M_2$ are infinite compact metric spaces such that the spaces $C(\prod_{i=1}^n K_i \times M_1)$ and $C(\prod_{i=1}^k L_i \times M_2)$ are isomorphic, then $n = k$ and the spaces $C(M_1)$ and $C(M_2)$ are isomorphic.
    \item If $M$ is a compact metric space such that the spaces $C(\prod_{i=1}^n K_i \times M)$ and $C(\prod_{i=1}^k L_i)$ are isomorphic, then $n = k$ and the space $C(M)$ is finite-dimensional or isomorphic to $c_0$.
\end{enumerate}
\end{theorem}

Let us recall some notation and simple facts regarding the spaces of vector-valued functions. For a compact space $K$ and a Banach space $(X,\|\cdot\|_X)$, 
by $C(K, X)$ we denote the space of continuous functions $f : K \to X$ with the supremum norm $\|f\| = \sup_{k \in K} \|f(k)\|_X$. 

\begin{fact}[\cite{Se71}*{Proposition 7.7.5}] \label{fact_C(KxL)_iso_C(C(K),C(L))}
If $K, L$ are compact spaces, then $C(K \times L) \simeq C(K, C(L))$.
\end{fact}

The next fact is obvious.

\begin{fact} \label{fact_C(K,X)_iso_C(K,Y)}
For any compact space $K$ and isomorphic Banach spaces $X, Y$, we have $C(K, X) \simeq C(K, Y)$.
\end{fact}

To fully classify the spaces of continuous functions on compact spaces of the family $\cK_\mathbf{A}$, we also need the following result.

\begin{lemma}[\textsf{BA}] \label{lem: C(I_A) isomorphic to c_0(C(I_A))}
If a set $A$ is $\omega_1$-dense in $(0, 1)$, then the Banach spaces $C(I_A)$ and $C(I_A,c_0)$ are isomorphic.
\end{lemma}

\begin{proof}
First, recall a simple observation that, for any compact space $K$, the space $C(K,c_0)$ is isometric to the $c_0$-sum of $\omega$ copies of $C(K)$. Indeed, $c_0$ is isometric to $C_0(\omega)$, hence $C(K,c_0)$ is isometric to $C(K,C_0(\omega))$, which in turn is isometric to $C_0(K\times \omega)$, see \cite{Se71}*{Proposition 7.7.6}. Clearly, this latter space is  isometric to the $c_0$-sum of $\omega$ copies of $C(K)$.

Let $a_0 = 1$ and $(a_n)_{n\ge 1}$ be a decreasing sequence in $A$ which converges to $0$. 
Put
\[K_n = [(a_{n+1}, 1), (a_n, 0)] \sub I_A\]
for $n\in \omega$ (an interval with respect to lexicographic order in $I_A$).
For every $n\in \omega$, the set $A \cap K_n$ is $\omega_1$-dense in $(a_{n+1}, a_n)$, so by Corollary \ref{cor: IA, IB are homeomorphic}, $C(K_n)$ is isometric to $C(I_A)$.
As $I_A$ is Fr\'echet, due to Fact \ref{pre: C(K) iso to its hyperplanes}, we have $C(I_A) \simeq C_0(I_A\| \{(0, 0)\})$.

What remains is a standard observation that $C_0(I_A\| \{(0, 0)\})$ is isometric to the $c_0$-sum of spaces $C(K_n), n\in\omega$.
\end{proof}

Now we have all the tools to prove the main result of this Section.

\begin{theorem}[\textsf{BA}] \label{thm_isomorphic_classes_of_finite_products}
Let $K$ be a nonmetrizable finite product of separable compact lines of weight $\le \omega_1$. Then the space $C(K)$ is isomorphic to exactly one space $C(L)$ for $L \in \cK_\mathbf{A}$.
\end{theorem}

\begin{proof}
By Theorem \ref{thm_different_C(K)_are_noniso}, for two distinct spaces $K_1, K_2 \in \cK_\mathbf{A}$, we have $C(K_1) \not\simeq C(K_2)$. 
Thus, it is enough to prove that, for any $m\ge 1$, and any separable compact lines $K_1, \dotsc, K_m$ of weight $ \leq \omega_1$, if at least one of them is nonmetrizable, then there is some $L \in \cK_\mathbf{A}$ such that $C(\prod_{i=1}^m K_i) \simeq C(L)$.

Without loss of generality, assume that for some $1 \leq n \leq m$, we have $w(K_i) = \omega_1 \iff i \leq n$. Then $M = \prod_{i=n+1}^m K_i$ is a metrizable compact space. 

By Facts \ref{fact_C(KxL)_iso_C(C(K),C(L))}, \ref{fact_C(K,X)_iso_C(K,Y)} we have
\[C\left(\prod_{i=1}^m K_i\right) \simeq C\left(\prod_{i=1}^n K_i, C\left(\prod_{i=n+1}^m K_i\right)\right) \simeq C\left(\prod_{i=1}^n K_i, C(M)\right).\]
Using Theorem \ref{thm: Under BA there is one C(K_A) space} we can also see that
\[C\left(\prod_{i=1}^n K_i\right) \simeq  C\left(\prod_{i=1}^{n-1} K_i, C(K_n)\right) \simeq  C\left(\prod_{i=1}^{n-1} K_i, C(I_\mathbf{A})\right) \simeq  C\left(I_\mathbf{A} \times \prod_{i=1}^{n-1} K_i\right),\]
so by continuing this process further we obtain $C(\prod_{i=1}^n K_i) \simeq C((I_\mathbf{A})^n)$.
It follows that
\[C\left(\prod_{i=1}^m K_i\right) \simeq C\big((I_\mathbf{A})^n \times M\big).\]
By the results of Bessaga, Pe\l czy\'nski \cite{BP60} and Miljutin \cite{Mi66} we know that $C(M)$ is isomorphic to $C(M')$, where
\[M' \in \{\{0, 1, \dotsc, n-1\} : n\in \omega\} \cup \{[0, 1]\} \cup \{\omega^{\omega^\alpha}+1 : 0 \leq \alpha < \omega_1\}.\]
By Facts \ref{fact_C(KxL)_iso_C(C(K),C(L))} and \ref{fact_C(K,X)_iso_C(K,Y)} we have
\[C\big((I_\mathbf{A})^n \times M\big) \simeq C\big((I_\mathbf{A})^n, C(M)\big) \simeq C\big((I_\mathbf{A})^n, C(M')\big) \simeq C\big((I_\mathbf{A})^n \times M'\big).\]
Additionally, if $C(M')$ is finite-dimensional (i.e., $M'$ is finite), then we have 
\begin{eqnarray*}
C\big((I_\mathbf{A})^n \times M'\big) &\simeq& C\big((I_\mathbf{A})^{n-1}\times I_\mathbf{A} \times M'\big) \simeq C\big((I_\mathbf{A})^{n-1},C(I_\mathbf{A} \times M')\big)\\ 
&\simeq& C\big((I_\mathbf{A})^{n-1},C(I_\mathbf{A})\big) \simeq C((I_\mathbf{A})^n)
\end{eqnarray*}
due to Theorem \ref{thm: Under BA there is one C(K_A) space} and the fact that the space $I_\mathbf{A} \times M'$ is homeomorphic to a separable compact line. 

If  $C(M')$ is isomorphic to $c_0$ (i.e., $M' = \omega+1$), then 
\begin{eqnarray*}
C((I_\mathbf{A})^n \times M') &\simeq& C((I_\mathbf{A})^{n-1}\times I_\mathbf{A} \times M') \simeq C((I_\mathbf{A})^{n-1},C(I_\mathbf{A} \times M'))\\ 
&\simeq& C((I_\mathbf{A})^{n-1},C(I_\mathbf{A},C(M'))) \simeq C((I_\mathbf{A})^{n-1},C(I_\mathbf{A},c_0))\\ 
&\simeq& C((I_\mathbf{A})^{n-1},C(I_\mathbf{A})) \simeq C((I_\mathbf{A})^n)
\end{eqnarray*}
due to Lemma \ref{lem: C(I_A) isomorphic to c_0(C(I_A))}.
\end{proof}

\section{Final remarks} \label{sec_remarks}
\subsection{Compact spaces of weight \texorpdfstring{$\kappa$}{κ}}

Theorem \ref{many} should be compared with the corollary from the Miljutin, Bessaga and Pełczyński's classification of separable Banach spaces of the form $C(K)$ which says that there are only $\omega_1$ mutually nonisomorphic such spaces regardless of the relation between $\omega_1$ and $2^\omega$ (although the spaces $\ell_p$ for $1<p<\infty$ form a family of cardinality $2^\omega$ consisting of mutually nonisomorphic separable Banach spaces).  
Our method of proving Proposition \ref{many} works only for regular cardinals.
In particular, we do not know the answer to the following:

\begin{question} Suppose that $\kappa$ is a singular cardinal. Is it provable  
without any additional hypothesis that there are $2^\kappa$ mutually nonisomorphic
Banach spaces of the form $C(K)$  (of any form) of density $\kappa$?
\end{question}

Note that the answer to the above questions is positive if we assume the Generalized Continuum Hypothesis (\textsf{GCH}).
Indeed, Kisljakov's classification of spaces of the form $C([0, \alpha])$ for
$\alpha<\kappa^+$ (\cite{Ki75}) shows that without any additional hypothesis we have $\kappa^+$ mutually nonisomorphic Banach spaces of the form $C(K)$ for any uncountable cardinal $\kappa$. 
As \textsf{GCH} means that $2^\kappa=\kappa^+$ for every infinite cardinal $\kappa$, we obtain the positive answer to the above questions under \textsf{GCH}.

Also note that we cannot have more than $2^\kappa$ mutually nonisometric (and so nonisomorphic) Banach spaces of the form $C(K)$.
This follows from the fact that all compact spaces of weight not exceeding $\kappa$ embed homeomorphically into $[0,1]^\kappa$ and this space has at most $2^\kappa$ closed subsets as they all are complements of some unions of elements of the basis. 
This also implies that there are at most $2^\kappa$ mutually nonisometric general Banach spaces of density $\kappa$. 
Indeed, each such space isometrically embeds into some space of the form $C(K)$ of density $\kappa$ and there are only $2^\kappa$ closed subspaces of such a $C(K)$ as they are again complements of unions of some open balls around the points in the dense set of cardinality $\kappa$. 

It should also be noted that for every uncountable cardinal $\kappa$ (including singular cardinals), there are $2^\kappa$ mutually nonhomeomorphic compact spaces of weight $\kappa$. 
This follows from the Stone duality and the existence of many examples of families of cardinalities $2^\kappa$ of mutually nonisomorphic Boolean algebras of cardinality $\kappa$ presented in \cite{Mo89}.
In particular, it is shown in Theorem 1.3 of \cite{Mo89} that for every uncountable $\kappa$ there are $2^\kappa$ mutually nonisomorphic interval Boolean algebras. 
This together with Theorem 15.7 from \cite{Ko89} implies that there are $2^\kappa$ mutually nonhomeomorphic linearly ordered compact spaces for each uncountable cardinal $\kappa$.

Let us also mention that the possibility of using families of pairwise disjoint stationary subsets of $\omega_1$ to build $2^{\omega_1}$
different structures of ``sizes'' $\omega_1$  as in the proof of Theorem \ref{many} has been well known. 
For example, this argument shows that there are many nonhomeomorphic 2-dimensional nonmetrizable manifolds (Examples 6.2 of \cite{Ny84}) or is used in \cite{To81} in the context of continua. 

\subsection{Separable compact lines}

The whole proof of Theorem \ref{thm: Under BA there is one C(K_A) space} relies on Baumgartner's Theorem (Theorem \ref{thm: All omega1-dense sets are isomorphic}). 
In fact, we do not need its full form, but only a weaker version stating that for two $\omega_1$-dense subsets $A, B$ of $(0, 1)$ spaces $C(I_A)$ and $C(I_B)$ are isomorphic. 
This statement may perhaps be consistently proven under weaker assumptions, for example, that every two $\omega_1$-dense subsets $A, B$ of $(0, 1)$ there is an order embedding of $A$ to $B$. 
Many interesting facts about this and similar axioms can be found in \cite[Section 6]{ARS85}.

The following comment is purely theoretical (as it is not known whether even \textsf{BA}$(\omega_2)$ is consistent), but can still be of some value (in particular, it indicates a slightly different way of proving Theorem \ref{thm: Under BA there is one C(K_A) space}).

Under the assumption of \textsf{BA}$(\kappa)$, for $\kappa$ of uncountable cofinality, it is possible to prove the analogue of Theorem \ref{thm: Under BA there is one C(K_A) space} for separable compact lines of weight $\kappa$.
The proof, involving Pełczyński decomposition, would use a counterpart of Lemma \ref{lem: C(I_A) isomorphic to c_0(C(I_A))}, for $\kappa$ instead of $\omega_1$, and the following two facts.

\begin{lemma}[\textsf{BA}$(\kappa)$]
For every separable compact line $L$ of weight $\kappa$ and every $\kappa$-dense set $B \sub (0, 1)$ the space $C(L)$ can be embedded in $C(I_B)$ as a complemented subspace.
\end{lemma}

\begin{proof}
By repeating the proof of Lemma  \ref{lem: reduction to omega_1-dense sets}, the space $C(L)$ is isomorphic to $C(L') \oplus C(N)$ for a $0$-dimensional perfect separable compact line $L'$ of weight $\kappa$ and some metrizable compact space $N$.
But by Lemma \ref{lem: C(KA) and C(KA) x C(M) are isomorphic}, we have 
\[C(L') \simeq C(L') \oplus C(N) \simeq C(L).\]
By Proposition \ref{prop_zero_dim_no_isolated_points}, the space $L'$ is homeomorphic to $C_A$ for some set $A \sub C$ of cardinality $\kappa$. 
Note that for $\kappa > \omega_1$ the analogue of Corollary \ref{cor: A simeq B, then I_A simeq I_B} also holds.
The set $(B\sm C) \cup A$ is still $\kappa$-dense in $(0, 1)$, so the spaces $I_B$ and $I_{(B\sm C) \cup A}$ are homeomorphic.
It is standard to check that $C_A$ is a closed subset of $I_{(B\sm C) \cup A}$, so by Lemma \ref{lem: decomposition through a closed subspace} we have
\[C(C_A) | C(I_{(B\sm C) \cup A}),\]
which ends the proof, as $C(L) \simeq C(L') \simeq C(C_A)$ and $C(I_{(B\sm C) \cup A}) \simeq C(I_B)$.
\end{proof}

\begin{lemma}[\textsf{BA}$(\kappa)$] If $\kappa$ has uncountable cofinality then,
for every separable compact line $L$ of weight $\kappa$ and every $\kappa$-dense set $B \sub (0, 1)$, the space $C(I_B)$ can be embedded in $C(L)$ as a complemented subspace.
\end{lemma}

\begin{proof} By Lemma \ref{lem: decomposition through a closed subspace}
it is enough to show that $L$ contains a topological copy of $I_B$. To achieve this we  need again to use some ideas from the proof of Lemma  \ref{lem: reduction to omega_1-dense sets}. 

As usual, we can assume that $L = I_A$, where $A$ is a subset of $(0,1)$ of size $\kappa$. Let 
\[U = \bigcup \{(a, b) \sub (0,1): a, b \in \bQ \mbox{ and } |(a, b) \cap A| < \kappa\}.\]
Clearly, $|U\cap A| < \kappa$ (by $cf(\kappa) >\omega$) and $A\sm U$ is $\kappa$-dense in a compact set $M = I \sm U$. We consider a separable compact line $K = M_{A\sm U}\subseteq L$. Let $\alpha = \min\{\beta: K^{(\beta+1)} = K^{(\beta)}\}$. In the same way as in the proof of Lemma \ref{lem: decomposition through a closed subspace} we can check that the subspace $K^{(\alpha)}$ of $L$ is homeomorphic to $I_B$.
\end{proof}

\bibliographystyle{amsalpha}
\bibliography{refs}

\end{document}